\documentclass[11pt, a4paper]{amsart}

\usepackage{amsmath, amsthm, amssymb,fullpage}

\usepackage{color}
\usepackage{xcolor}

\usepackage{enumitem}
\usepackage{hyperref}

\theoremstyle{plain}
\newtheorem{theorem}{Theorem}[section]
\newtheorem{corollary}[theorem]{Corollary}
\newtheorem{lemma}[theorem]{Lemma}
\newtheorem{proposition}[theorem]{Proposition}

\theoremstyle{definition}
\newtheorem{definition}[theorem]{Definition}

\theoremstyle{remark}
\newtheorem{remark}[theorem]{Remark}

\numberwithin{equation}{section}

\newcommand{\C}{\mathbb C}
\newcommand{\R}{\mathbb R}

\newcommand{\Cc}{\mathcal C}

\usepackage{bbold}

\usepackage{tikz-cd}

\DeclareMathOperator{\dist}{dist}

\usepackage{stmaryrd}

\usepackage{enumitem}
\usepackage{hyperref}

\usepackage{mathrsfs}

\title[Exponential mixing of all orders and CLT for automorphisms of compact K\"ahler manifolds]{Exponential mixing of all orders and CLT for automorphisms of compact K\"ahler manifolds}

\begin{author}[F.~Bianchi]{Fabrizio Bianchi}
\address{ 
CNRS, Univ. Lille, UMR 8524 - Laboratoire Paul Painlev\'e, F-59000 Lille, France}
  \email{fabrizio.bianchi$@$univ-lille.fr}
\end{author}

\begin{author}[T.C.~Dinh]{Tien-Cuong Dinh}
\address{National University of Singapore, Lower Kent Ridge Road 10,
Singapore 119076, Singapore}
\email{matdtc$@$nus.edu.sg }
\end{author}



\begin{document} 

\hyphenpenalty=10000

\maketitle

\begin{abstract}
We consider the unique measure of maximal entropy of an automorphism of
a compact K\"ahler manifold with simple action on cohomology. We show that it is
exponentially mixing of all orders with respect to H\"{o}lder observables. It follows that the Central Limit Theorem
(CLT) holds for these observables. 
 In particular,
our result applies to all automorphisms of
compact K\"ahler surfaces with positive entropy.
\end{abstract}

\bigskip

\noindent
{\bf Notation.} 
The pairing $\langle\cdot, \cdot\rangle$
 is used for the integral of a function with respect to a measure or more
generally the value of a current at a test form.
 By
 $(p,p)$-currents 
 we mean currents
 of bi-degree $(p,p)$.
The notations $\lesssim$ and $\gtrsim$ 
stand for inequalities up to
a multiplicative constant.
If $R$ and $S$
are two real currents of the same bi-degree, we write $|R|\leq S$ when
$S\pm R\geq 0$.
 Observe that this forces $S$ to be positive.

Given a compact K{\"a}hler manifold $X$ of dimension $k$, for every $0\leq q \leq k$ 
we will denote by $\mathscr D_q(X)$ (resp. $\mathscr D_q^0 (X)$)
the real space
generated by positive
closed (resp. $dd^c$-exact) 
$(q,q)$-currents on $X$. For $S\in \mathscr D_q(X)$,
 we will denote by $\{S\}$ the cohomology class of $S$ in $H^{q,q}(X,\R)$.

\section{Introduction}

Let $(X,\omega)$ be a compact K\"ahler manifold
of dimension $k$
and $f$
a holomorphic automorphism of $X$. 
We refer to \cite{BK09,FT23,McM06,McM07,Oguiso09,OT15} 
for interesting examples of such maps, and to
 \cite{DTD12,Dinh05JGA,DS05, DS10JAG}
for their general properties, see also
\cite{Cantat01,CD20,DH22,FT21,Guedj10,Oguiso14}.
We denote by $f^n$ the
iterate of order $n$ of $f$.
For $0\leq q \leq k$,
the dynamical degree $d_q$
is defined as the spectral radius
of the pull-back operator acting on the cohomology group
$H^{q,q}(X,\R)$. We have $d_0=d_k=1$ and $d_q (f^n)= d_q (f)^n$
for all $n\in \mathbb N$.

By a 
fundamental result of
 Khovanskii \cite{Khovanskii79}, Teissier \cite{Teissier79}, and Gromov 
\cite{Gromov90},
the sequence $q\mapsto \log d_q$
is concave, see also \cite{DN06}.
This implies that there exist integers $0\leq p\leq p'\leq k$ such that
\[
1=d_0 < d_1 < \ldots < d_p = \ldots = d_{p'} > \ldots > d_{k}=1.
\]
We say that $f$ has \emph{simple action
on cohomology} 
if
$p=p'$
 and if moreover the action of $f^*$
on $H^{p,p}(X,\R)$
 admits a unique eigenvalue of maximal
 modulus. 
 Such eigenvalue is then necessarily equal to $d_p$.
 We denote in this case 
 by 
 $\delta=\delta(f)$ the maximum between $\max_{q\neq p} d_q$ and the 
moduli
 of 
  the other eigenvalues for the action of $f^*$ on
  $H^{p,p} (X,\R)$.
  We call $d_p$ the \emph{main dynamical degree} 
  and $\delta$ the \emph{auxiliary dynamical degree}
  of $f$.

 \medskip
 
From now on, we assume that
 $f$ has simple action on cohomology.
 It admits a unique probability
 measure of maximal entropy $\mu$, which
  is the intersection
of 
a 
 positive closed $(p,p)$-current $T_+$
 and a 
positive closed 
$(k-p,k-p)$-current $T_-$
(the main Green currents of $f$),
 see \cite{DS05,DS10JAG}.
Such measure is also called the \emph{equilibrium measure} of $f$, and is mixing
and hyperbolic.
It was shown
 in \cite{DS10}
that such measure is exponentially mixing
for H{\"o}lder observables, 
see also
\cite{Dinh05,Vigny15,Wu22}.
 We consider here the following stronger property.

\begin{definition}\label{d:exp-mixing-all-orders}
Let $\nu$ be 
an invariant measure 
and  $(E,\|\cdot\|_E)$ a normed space
of real functions on $X$, with $\|\cdot\|_{L^r (\nu)}\lesssim \|\cdot\|_E$ for all $1\leq r <\infty$.
We say that
 $\nu$ is \emph{exponentially mixing of all orders} 
for observables in $E$ 
 if 
for all $\kappa \in \mathbb N^*$ there exist
constants
$C_{\kappa}>0$
and
 $0< \theta_{\kappa} < 1$
  such that,  for all observables
 $g_0, \dots, g_{\kappa}\in E$
 and 
 integers
 $0=: n_0 \leq n_1 \leq \dots \leq n_{\kappa}$, 
 we have
\[
\Big|
\langle \nu,
 g_0 ( g_1 \circ  f^{n_1} )
\ldots
 (g_{\kappa} \circ f^{n_{\kappa}})
\rangle
- \prod_{j=0}^{\kappa} \langle  \nu, g_j
\rangle
\Big|
\leq C_{\kappa}
\cdot \Big(
\prod_{j=0}^{\kappa} 
\|g_j\|_{E}
\Big) \cdot
\theta_{\kappa}^{\min_{0\leq j \leq \kappa-1} ( n_{j+1}- n_j)}.
\] 
\end{definition}

The exponential mixing of all orders is one of the strongest ergodic
properties for a dynamical system.
It implies a number of statistical properties that seem unattainable without such quantitative control, see for instance
 \cite{BG, DFL}.
It is a main open question if this is implied by the exponential mixing of order 1, see for instance \cite[Question 1.5]{DKRH}.
 We established in \cite{BD23} the exponential mixing of all orders for every complex H\'enon map.
The following is the
main result of this paper.

\begin{theorem}\label{t:intro-mixing}
Let $f$ be a holomorphic automorphism of a compact K\"ahler
manifold
$(X,\omega)$
with simple action on cohomology.
Let $d_p$ be its 
main
 dynamical degree
 and 
  $\delta$ 
  its auxiliary dynamical degree.
 Then, for every $\delta<\delta'<d_p$ and  $0<\gamma\leq 2$,
 the equilibrium probability measure
$\mu$ of $f$ is exponentially mixing of all orders $\kappa \in \mathbb N^*$
for all
$\Cc^\gamma$
 observables,
 with $\theta_\kappa= (d_p/\delta')^{-(\gamma/2)^{\kappa+1} /2}$,
 see Definition \ref{d:exp-mixing-all-orders}.
\end{theorem}

In order to prove Theorem \ref{t:intro-mixing} we rely on delicate estimates
coming from pluripotential theory and on the theory of positive closed currents.
We partially follow the strategy in \cite{BD23}, but
the absence of plurisubharmonic functions on compact K{\"a}hler
manifolds
makes it not possible to employ a
key step
developed there.
Instead, we rely on the theory of \emph{super-potentials}
for positive closed currents,
and on quantitative estimates on the convergence of sufficiently regular positive closed
currents towards the Green currents. We will
 give below an overview of our strategy.

\medskip

As in 
\cite{BD23},
the following is 
a consequence of Theorem \ref{t:intro-mixing}
and \cite{BG}.
 Recall that 
 $u$ \emph{satisfies the Central Limit Theorem}
(CLT)
with variance $\sigma^2 \geq 0$ with respect to $\mu$
if
 $n^{-1/2} (S_n (u)  - n\langle\mu,u\rangle )\to \mathcal N(0,\sigma^2)$ in law,
 where $\mathcal N (0,\sigma^2)$
denotes the
(possibly degenerate, for $\sigma =0$)
 Gaussian distribution 
with mean 0 and variance $\sigma^2$, i.e., 
 for any interval $I \subset \mathbb R$ we have
\[
\lim_{n\to \infty} \nu \Big\{
\frac{S_n (u) - n\langle\mu,u\rangle }{\sqrt{n} } \in I
\Big\}
=
\begin{cases}
1 \mbox{ when } I \mbox{ is of the form } I=(-\delta,\delta)  & \mbox{ if } \sigma^2=0,\\
\frac{1}{\sqrt{2\pi\sigma^2}}\displaystyle\int_I e^{-t^2 / (2\sigma^2)} dt &  \mbox{ if }
\sigma^2 >0.
\end{cases} \]

\begin{corollary}\label{t:intro-clt-manifold}
Let $f$ be a holomorphic automorphism
of a compact K\"ahler
manifold
$(X,\omega)$
with simple action on cohomology.
Then all H\"older
observables $u \colon X\to \R$
satisfy the Central Limit Theorem
with respect to the measure of maximal entropy $\mu$
of $f$ with
\[\sigma^2
=
\sum_{n\in \mathbb Z} \langle\mu, \tilde u (\tilde u \circ f^n) \rangle
=
\lim_{n\to \infty}
\frac{1}{n}\int_{X}
(\tilde u+ \tilde u \circ f + \ldots + \tilde u \circ f^{n-1})^2 d\mu,
\]
where $\tilde u := u-\langle\mu,u\rangle$.
\end{corollary}

Let now $X$ be a compact
K\"ahler surface and 
$f$ 
 an automorphism of positive entropy.
By \cite{Gromov03,Yomdin87}, the topological entropy is equal to $\log d_1$, see also \cite{DS04}.
In particular, $d_1$ is strictly 
larger than $1$.
A result by Cantat 
\cite{Cantat01}
says that all the eigenvalues of 
the action of $f^*$ on $H^{1,1}(X,\R)$ have modulus 1, except
for  two eigenvalues $d_1$ and $1/d_1$, which have multiplicity 1.
In particular, every
automorphism of positive entropy of a K\"ahler surface has simple action on cohomology.

\begin{corollary}\label{c:intro-clt-surface}
Let $f$ be a holomorphic automorphism of positive entropy on
a compact K\"ahler surface $X$.
 Then, for every $1<d'<d_1$ and $0<\gamma\leq 2$,
 the equilibrium probability measure
$\mu$ of $f$ is exponentially mixing of all orders
$\kappa \in\mathbb N^*$
for all
$\Cc^\gamma$
 observables, 
with 
$\theta_\kappa = (d')^{-(\gamma/2)^{\kappa+1} /2}$,
see 
Definition \ref{d:exp-mixing-all-orders}.
Moreover,
 all H\"older
observables satisfy the Central Limit Theorem
with respect to the measure of maximal entropy of $f$.
\end{corollary}

\subsection*{Strategy of the proof of Theorem \ref{t:intro-mixing}}

Using the classical
theory of interpolation \cite{Triebel95}, it is enough to prove 
the theorem in the case where $\gamma=2$. 
Consider the compact K{\"a}hler manifold $\mathbb X:= X\times X$. The automorphism
$F:= (f,f^{-1})$ of $\mathbb X$
and
 its inverse
  have simple action on cohomology, with the largest dynamical degree being that of order $k$, which is
equal to $d^2_p$. We can define the main Green currents $\mathbb T_+$ and $\mathbb T_-$
for $F$ as $\mathbb T_+ := T_+ \otimes T_-$ and $\mathbb T_- := T_-\otimes T_+$. They satisfy
$(F^n)^* (\mathbb T_+) = d^2_p \mathbb T_+$
and $(F^n)_* (\mathbb T_-) = d^2_p \mathbb T_-$. Proving the exponential mixing of order $\kappa$
for the $\kappa+1$ observables
$g_0, \dots, g_\kappa$ can be reduced to proving
the estimate
\begin{equation}\label{e:goal-mixing}
|\langle d_p^{-n_1} (F^{n_1/2})_* [\Delta] - \mathbb T_- ,  \Theta_{\{g_j\}, \{n_j\}} \rangle|
\lesssim  \prod_{j=0}^\kappa \|g_j\|_{\Cc^2}
 d^{- \min_{0 \leq j \leq \kappa-1} (n_{j+1}- n_j)/2},
\end{equation}
where  $[\Delta]$ denotes the current of integration on the diagonal $\Delta \subset X\times X$,
$(z,w)$ the coordinates on $X\times X$, we set
\[
\Theta_{\{g_j\}, \{n_j\}} := g_0 (w) g_1 (z) (g_2 \circ f^{n_2-n_1}(z)) \ldots (g_\kappa \circ f^{n_\kappa- n_1}(z)) \mathbb T_+,
\]
and we assumed for simplicity that 
$n_1$ is even.

\medskip

In the case of Hénon maps, we established in \cite{BD23}
the above convergence
by proving that $\Theta_{\{g_j\}, \{n_j\}}$
can be replaced by suitable currents
$\Theta^{\pm}$ with $dd^c \Theta^{\pm}\geq 0$, for which the estimate above can be proved
thanks to the properties of plurisubharmonic functions. As non-trivial 
plurisubharmonic functions
 do not exist
on compact K{\"a}hler manifolds, that approach cannot work here. Instead, we use more
refined estimates on the regularity of the currents involved. 
We prove that 
if $\Theta_{\{g_j\}, \{n_j\}}$
 is \emph{H{\"o}lder continuous} in a precise sense (i.e., when seen as a function on the 
space of
positive exact $(k,k)$-currents,
endowed with a suitable metric), then the convergence \eqref{e:goal-mixing} holds.
This is done by exploiting the theory of \emph{super-potentials}
 for positive closed currents, as developed
by Sibony and the second author.

\medskip

The main task becomes to prove the H{\"o}lder continuity  of the current $\Theta_{\{g_j\}, \{n_j\}}$.
Observe that the estimate needs to be uniform in
the $n_j$'s and in the $g_j$'s (assuming $\|g_j\|_{\Cc^2}\leq 1$ for all $j$), 
in order for the implicit constant in \eqref{e:goal-mixing} not to depend on such parameters.
Observe also that this problem does not exist when just proving the mixing of order $\kappa=1$, see \cite{DS10}. 
This is the main technical point of the current paper.

\medskip

By means of a general comparison principle for the
super-potentials of positive closed currents \cite{DNV18}, 
we show that it is enough to find a positive closed current $\Xi$
with a H{\"o}lder continuous super-potential and such that 
\[
|dd^c \Theta_{\{g_j\}, \{n_j\}}|\leq \Xi
\quad 
\mbox{ for all } n_j 
\mbox{ and all }
g_j \mbox{ with } \|g_j\|_{\Cc^2} \leq 1.
\
\]
Finding such $\Xi$
and establishing such an estimate
 rely on the 
  gap between $d_p$ and
 the 
 auxiliary dynamical degree of $f$
  and 
on H{\"o}lder estimates for the action on  $f^*$ on $(p+1, p+1)$-currents,
that we also develop in this paper.

\subsection*{Acknowledgments}
We would like to thank the National University of Singapore, the Institut
de Math{\'e}matiques de Jussieu-Paris Rive Gauche, IMPAN, 
Xiaonan Ma, and Feliks Przytycki 
for the warm welcome and
the excellent work conditions.

This project has received funding from
 the French government through the Programme
 Investissement d'Avenir
 (LabEx CEMPI /ANR-11-LABX-0007-01,
ANR QuaSiDy /ANR-21-CE40-0016,
ANR PADAWAN /ANR-21-CE40-0012-01)
and the NUS
and MOE through the grants
A-0004285-00-00
and 
MOE-T2EP20120-0010.
This work is also partially supported by the Simons Foundation Award No. 663281 granted to the Institute of Mathematics of the Polish Academy of Sciences for the years 2021-2023.

\section{Super-potentials of currents on compact K{\"a}hler manifolds}\label{s:super-potentials}

We fix in this section a $k$-dimensional
compact
K\"ahler manifold $X$ and a K\"ahler form $\omega$ on $X$.
For $0\leq q\leq k$, we denote by $\mathscr D_q(X)$ 
the real space
generated by positive closed $(q,q)$-currents on $X$, and by $\mathscr D^0_q (X)$
the subspace
of $\mathscr D_q(X)$
given by the currents whose cohomology class in $H^{q,q}(X,\R)$
vanishes. By the $\partial \bar \partial$-lemma, this is the set
of currents in $\mathscr D_q(X)$
which are $\partial \bar \partial$-exact.
 Since $X$ is fixed, for simplicity
in this section
we will drop the dependence on $X$ for these spaces and denote them by
$\mathscr D_q$ and $\mathscr D_q^0$.
We set $h_q := \dim H^{q,q}(X,\R)$
and fix a collection $\alpha_q = (\alpha_{q,1}, \dots, \alpha_{q,h_q})$ 
of real smooth $(q,q)$-forms on $X$ such that the family of cohomology classes
$\{\alpha_{q,j}\}$
is a basis for $H^{q,q} (X,\R)$. We also 
choose a collection of smooth real closed $(k-q,k-q)$-forms
 $\check\alpha_q
 =
(\check\alpha_{q,1},
 \dots, \check\alpha_{q,h_q})$
which represent
 a dual basis of $\alpha_q$  in $H^{k-q, k-q}(X,\R)$
 for the 
 Poincaré duality.

\medskip

Recall that the mass of a positive closed current $S\in H^{q,q}(X,\R)$
 is defined as $\|S\|:= \langle S, \omega^{k-q} \rangle$ and it depends only on the cohomology
 class $\{S\}$ of $S$ in $H^{q,q}(X,\R)$.
The norm $\|\cdot\|_*$ is defined for $S\in \mathscr D_q$
as 
\[
\|S\|_* :=\min\{ \|S'\|\colon |S|\leq S'\}.
\]
When $S$ is $dd^c$-exact
the norm $\|S\|_*$ is equivalent to the norm given by
$\min \|S^+\|$, where
the minimum is taken over all decompositions
$S=S^+-S^-$ with $S^{\pm}$
positive closed.
Observe that, for each such decomposition, we have $\|S^+\|= \|S^-\|$
as a consequence of the equality
$\{S^+\}=\{S^-\}$. We say that a subset of $\mathscr D_q$
is $*$-\emph{bounded}
if it is bounded for $\|\cdot\|_*$.

\medskip

\begin{definition}\label{d:star}
We say that a sequence $S_n \in \mathscr D_q$ $*$-\emph{converges to} $S\in\mathscr  D_q$
if $S_n\to S$ in the sense of currents, and $\|S_n\|_*$ is bounded by a constant independent of $n$.
\end{definition}

\begin{remark}
The convergence 
is Definition
\ref{d:star}
defines a topology on $\mathscr D_q$,
which is not the one given by $\|\cdot\|_*$.
We can also 
define
 a topology on $\mathscr D^0_q$
with the same definition.
By
\cite{DS04},
smooth forms are dense in $\mathscr D_q$ and $\mathscr D^0_q$
with respect to  the topology of Definition \ref{d:star}.
\end{remark}

For any $0<l<\infty$, denote by $\|\cdot\|_{\Cc^l}$
the
standard $\Cc^l$ norm on the space of differential forms.
We consider a norm $\|\cdot\|_{\Cc^{-l}}$
and a
 distance $\dist_l$ on $\mathscr D_{q}$
 defined by
 \[
 \|S\|_{\Cc^{-l}}:= \sup_{\|\Phi\|_{\Cc^l}\leq 1}|\langle S,\Phi\rangle|
 \quad \mbox{ and }
 \quad
 \dist_l (S,S')
 := \|S-S'\|_{\Cc^{-l}},
 \]
where the supremum in the first definition
is on smooth 
$(k-q,k-q)$-forms 
$\Phi$ on $X$. Observe that,
by
 interpolation
 \cite{Triebel95},
 for every $0<l<l'<\infty$ and $m>0$,
 we have
\begin{equation}\label{e:holder}
\|S\|_{\Cc^{-l'}}\leq \|S\|_{\Cc^{-l}}\leq c_{l,l',m} \|S\|_{\Cc^{-l'}}^{l/l'}
\quad \mbox{ for all } S  \mbox{ such that }
 \|S\|_* \leq m,
\end{equation}
for some positive constant $c_{l,l',m}$.
In particular, 
 the above inequalities imply that
\[
\dist_{l'}\leq \dist_l\leq c_{l,l',m} ( \dist_{l'})^{l/l'}
\quad \mbox{ on } \quad 
\{S\in \mathscr D_p \colon \|S\|_* \leq m/2\}.
\]

\medskip

We now
 recall the notion of \emph{super-potential} for currents in $\mathscr D_q$, see \cite{DS10JAG}.
 This notion generalizes
that
 of potentials for $(1,1)$-positive closed currents on $X$, which are
 the quasi-plurisubharmonic functions on $X$.
Super-potentials
of a current $S\in \mathscr D_q$
that we use here are
  functions,
 depending linearly on $S$, and defined on a subset of
 $\mathscr D_{k-q+1}^0$
 \footnote{Super-potentials can be defined on
 a subset of $\mathscr D_{k-q+1}$, but it is simpler to use $\mathscr D^0_{k-q+1}$,
 and we will only need this definition in this paper.}.
 \medskip

Take $R\in \mathscr D_{k-q+1}^0$. As $\{R\}=0$, we have $R= dd^c U_R$
for some $(k-q,k-q)$-current $U_R$, that we call a \emph{potential} of $R$.
By adding to $U_R$ a suitable combination of the 
$\check\alpha_{q,j}$'s,
one can assume that
$U_R$ is $\alpha_q$-{\emph normalized}, i.e., that
$\langle U_R,\alpha_{q,j}\rangle=0$ for all $1\leq j\leq h_q$. 
The \emph{$\alpha_q$-normalized super-potential} $\mathcal U_S$ of $S$ is defined as
\begin{equation}\label{e:def-super-potential}
\mathcal U_S (R) :=\langle S,U_R\rangle 
\end{equation}
whenever the RHS of the above
expression is well-defined. This is the case for instance when 
$S$ is smooth
 or when $R$ is smooth and we choose  $U_R$ 
 smooth.

\begin{lemma}\label{l:super-1}
Let $S,R,U_R$ be as above and such that
either $R$ or $S$ is smooth.
 \begin{itemize}
 \item[{\rm (i)}]
  The $\alpha_q$-normalized super-potential
  $\mathcal U_S$ of $S$ does not depend on the choice of the
  $\alpha_q$-normalized 
  potential $U_R$; in particular, it does not depend on the choice of the $\check\alpha_{q,j}$'s;
\item[{\rm (ii)}] 
If $\{S\}=0$, then $\mathcal U_S$ does not depend on the choice of $\alpha_q$.
\item[{\rm (iii)}]
If $S$ is a linear combination of the $\alpha_{q,j}$'s, then the $\alpha_q$-normalized
super-potential $\mathcal U_S$ of $S$ vanishes identically on $\mathscr D^0_{k-q+1}$.
\end{itemize}
\end{lemma}

\begin{proof}
(i)
Let $U_R$ and $U'_R$ be two $\alpha_q$-normalized potentials of $R$. As $dd^c (U_R-U'_R)=R-R=0$, 
by Poincaré duality and the $\partial\bar\partial$-lemma the class $\{U_R-U'_R\}\in H^{k-q,k-q}(X,\R)$
is well defined
and $\langle S, U_R-U_{R'}\rangle$
only depends on the cohomology class of $S$.
Since both $U_R$ and $U'_R$ are $\alpha_q$-normalized,
the cup-products
of such class with all the classes $\{\alpha_{q,j}\}$
 are equal to $0$. As the $\{\alpha_{q,j}\}$'s
  form a basis of 
$H^{q,q} (X,\R)$,
it follows that $\langle S,U_R- U'_R\rangle=0$, as required.

\medskip

(ii) As $\{S\}=0$, we have $S= dd^c U_S$ for some $(q-1,q-1)$-current $U_S$. For any potential $V_R$ of
 $R$, non necessarily
$\alpha_q$-normalized, we have
\begin{equation}\label{e:SR1}
\langle S, V_R\rangle =\langle dd^c U_S, V_R\rangle = \langle U_S,  dd^c V_R\rangle = \langle U_S, R\rangle.
\end{equation}
In particular, the first term of the above expression does not depend on
the choice of $V_R$, as this is the case for the last term.
The assertion follows.

\medskip

(iii) 
By definition, we have $\langle U_R, \alpha_{q,j}\rangle=0$ for all $j$. As $S$ is a linear combination of the
$\alpha_{q,j}$'s, we have $\mathcal U_S (R) =\langle S, U_R\rangle=0$. The assertion follows.
\end{proof}

\begin{remark}
Observe that, on the other hand, the value $\langle S,U_R\rangle$ and, by consequence, the $\alpha_q$-normalized
super-potential $\mathcal U_S$ of $S$, depend on
 $\alpha_q$ when $S$ is not exact.
\end{remark}

\begin{definition}
We say that $S\in \mathscr D_q$
has a \emph{continuous super-potential} if
 $\mathcal U_S$ extends continuously to $\mathscr D_{k-q+1}^0$, with respect to the topology given by Definition \ref{d:star}.
 We also use the notation $\mathcal U_S$ for the extended super-potential in this case.
 \end{definition}

\begin{proposition}\label{p:super-pot-1}
Take $q,q'$ with $q+q'\leq k$,
$S\in \mathscr D_q$,
$S'\in \mathscr D_{q'}$,
 and $R \in \mathscr D^0_{k-q+1}$.
\begin{enumerate}
\item[{\rm (i)}] If $S$ is smooth, then it has a continuous $\alpha_q$-normalized super-potential
for every choice of $\alpha_q$;
\item[{\rm (ii)}]
If $S$ has a continuous $\alpha_q$-normalized
super-potential, it also has continuous $\alpha'_q$-normalized
 super-potentials for every other normalization $\alpha'_q =(\alpha'_{q,1}, \dots, \alpha'_{q,h_q})$.
\item[{\rm (iii)}] 
If $S$ 
belongs to  $\mathscr D^0_q$
 and has 
continuous super-potentials
and $R$ is smooth we have
\[\mathcal U_S (R)=\mathcal U_R (S),\]
where $\mathcal U_R$ is the super-potential of $R$
(which is also independent of the choice of the
normalization);
\item[{\rm (iv)}]
If $S'$ has continuous super-potentials, then the current $S\wedge S'$ is well defined and depends continuously
on $S$.
\end{enumerate}
\end{proposition}

In particular, by the third
 item,
 for every $R\in \mathscr D_{k-q+1}^0$
 we can define
 \[
\mathcal U_R (S) := \mathcal U_S (R). 
 \]
 when $S\in \mathscr D_q^0$ has a continuous super-potential.

\begin{proof}

(i)
This is clear from the definition of the super-potential $\mathcal U_S$.

\medskip

(ii)
By (i) we can add to $S$ a smooth closed form. This does not change the problem.
So, we can assume that $\{S\}=0$. By Lemma \ref{l:super-1}(ii), $\mathcal U_S$ 
is independent of
the normalization. The assertion follows.

\medskip

(iii) Since 
$R$ is smooth,
the RHS in \eqref{e:def-super-potential} is well defined for every 
$S \in \mathscr D^0_p$.
As $\{S\}=0$, by Lemma \ref{l:super-1}(ii), $\mathcal U_S$ does not depend on the normalization 
and \eqref{e:SR1} holds. As the last term of that expression is equal to  $\mathcal U_R (S)$, the assertion follows.

\medskip

(iv)
This is a consequence of \cite[Theorem 3.3.2]{DS10JAG}.
\end{proof}

\begin{definition}\label{d:holder}
Take $S\in \mathscr D_q$.
For
 $l>0$, $0 < \lambda \leq 1$, and $M\geq 0$,
we say that a super-potential
$\mathcal U_S$ of $S$
is $(l,\lambda, M)$-\emph{H\"older continuous}
if it is continuous and we have
\[
|\mathcal U_S (R)|\leq M \|R\|_{\Cc^{-l}}^\lambda
\mbox{ for every } R\in \mathscr D^0_{k-q+1} \mbox{ with }\|R\|_*\leq 1.
\]
\end{definition}

\begin{lemma}\label{l:holder}
Take $S \in \mathscr D_q$ and $R,R'\in \mathscr D^0_{k-q+1}$ with $\|R\|_*\leq 1$ and $\|R'\|_* \leq 1$.
\begin{enumerate}
\item[{\rm (i)}]
If $S$ is smooth, then 
the $\alpha_q$-normalized
super-potential
 $\mathcal U_S$ of $S$
  is $(2,1,M)$-H{\"o}lder continuous 
  with $M \leq c \|S\|_{\Cc^2}$ for some constant $c>0$
  independent of $S$ (but possibly depending on $\alpha_q$).
\item[{\rm (ii)}]
 If $S$ has
  an
$(l, \lambda, M)$-H{\"o}lder continuous  
  $\alpha_q$-normalized
  super-potential $\mathcal U_S$, 
   then
\[|\mathcal U_S (R)- \mathcal U_S (R')|\leq 2^{1-\lambda} M \|R-R'\|^\lambda_{\Cc^{-l}}.\]
In particular, $\mathcal U_S$ can be seen as a H{\"o}lder continuous  function on
$\{R \in \mathscr D^0_{k-q+1}\colon \|R\|_*\leq 1\}$.
  \item[{\rm (iii)}] 
  If $S$ has
  an
$(l, \lambda, M)$-H{\"o}lder continuous  
  $\alpha_q$-normalized
  super-potential $\mathcal U_S$, 
  then it
  also has 
  $(l,\lambda, M')$-H{\"o}lder continuous 
  $\alpha'_q$-normalized
 super-potentials for every 
 choice of $\alpha'_q$ and
 for some $M'$ independent of $S$.
 \end{enumerate}
\end{lemma}

\begin{proof}
(i)  
By Lemma \ref{l:super-1}(iii), we can add to $S$ a combination of the $\alpha_{q,j}$'s
and assume that $\{S\}=0$. By Lemma \ref{l:super-1}(ii), we have
$\mathcal U_S (R)= \langle U_S,R\rangle$, where $U_S$ is a smooth 
potential of $S$. The assertion is now clear.

\medskip

(ii) As 
$R-R'\in \mathscr D^0_{k-q+1}$ and
$\|R-R'\|_*\leq 2$, the assertion follows from Definition \ref{d:holder}
applied with $(R-R')/2$ instead of $R$.

\medskip

(iii) 
As in Proposition \ref{p:super-pot-1}(ii), 
by (i) we can assume that $\{S\}=0$. The assertion now follows from Lemma \ref{l:super-1}(ii).
\end{proof}

\begin{proposition}\label{p:holder}
Take $q,q'$ with $q+q'\leq k$.
Let $\alpha_q, \alpha_{q'}, \alpha_{q+q'}$
be as above and
take $S,T\in \mathscr D_q$
and $S' \in \mathscr D_{q'}$
with $\|S\|_*\leq 1$, 
$\|T\|_*\leq 1$,
and $\|S'\|_*\leq 1$.
Assume that the
$\alpha_q$-normalized
super-potential
$\mathcal U_S$
of $S$
 is $(l,\lambda,M)$-H\"older
 continuous
 and that
  the
$\alpha_{q'}$-normalized
super-potential
$\mathcal U_{S'}$
of $S'$
 is $(l',\lambda',M')$-H\"older
 continuous.
\begin{enumerate}
\item[{\rm (i)}]  
For every $l_1>0$,
 $\mathcal U_S$ is
$(l_1,\lambda_1,M_1)$-H\"older continuous, for some
$\lambda_1$ and $M_1$ depending on $l$, $l_1$, $\lambda$, $M$, and $\alpha_q$,
 but independent of $S$.
\item[{\rm (ii)}]
If $S,T$ are positive and $T\leq S$, then 
any
$\alpha_q$-normalized
 super-potential 
$\mathcal U_{T}$ of $T$
is 
$(2,\lambda_1,M_1)$-H\"older continuous, for some
$\lambda_1$ and $M_1$ depending on $l$, $\lambda$, $M$,
 and $\alpha_q$,
but independent of $S$ and $T$.
\item[{\rm (iii)}]
$S\wedge S'$ has 
a
$(2,\lambda_1,M_1)$-H{\"o}lder continuous $\alpha_{q+q'}$-normalized
super-potential,
for some
$\lambda_1$ and $M_1$ depending on $l$, $l'$,
$\lambda$, $\lambda'$,
$M$, $M'$,
$\alpha_q$, $\alpha_{q'}$, 
and $\alpha_{q+q'}$,
but independent of $S$ and $S'$.
\end{enumerate}
\end{proposition}

\begin{proof}

The first assertion is a consequence of \eqref{e:holder}.
 The second
one
is a consequence of \cite[Theorem 1.1]{DNV18} and the first one.
The third assertion is a consequence of \cite[Proposition 3.4.2]{DS10JAG} and the first one.
\end{proof}

\section{Dynamical Green currents}\label{s:green}

We keep in this section the notations of Section \ref{s:super-potentials}.

\subsection{Convergence towards Green currents}
We fix here 
an automorphism $f$ of $X$ with simple action on cohomology.
We let 
 $p$ be such that $d_p$
 is the largest dynamical
  degree of $f$,
   and let
 $\delta$ be its auxiliary degree.
  We also fix
  $\delta'$ such that $\delta < \delta' < d_p$.
By assumption, the eigenspace associated to
the eigenvalue
 $d_p$ of the action of $f^*$ 
  on $H^{p,p}(X,\R)$ is a real line $H^+$.
 A \emph{Green $(p,p)$-current} $T_+$ of $f$
 is a non-zero positive closed $(p,p)$-current
 invariant under
  $d_p^{-1}f^*$, i.e., satisfying $f^* (T_+) = d_p T_+$. The cohomology class
  $\{T_+\}$ generates  $H^+$.

\begin{lemma}\label{l:T+}
Let $f$, $d_p$, $T_+$, $\alpha_p$
be as above.
\begin{enumerate}
\item[{\rm (i)}]
$T_+$  is the unique positive closed $(p,p)$-current
 in $\{T_+\}$, and it has a
 $(2,\lambda, M)$-H{\"o}lder continuous 
 $\alpha_p$-normalized
 super-potential for some $\lambda$ and $M$.
 \item[{\rm (ii)}]
 For every $S\in \mathscr D_p$ we have
 $d_p^{-n} (f^n)^* S\to sT_+$.
 Here the constant $s$ depends linearly on $\{S\}$. More precisely, $s$ is the constant
 such that that $d_p^{-n} (f^n)^*\{S\}\to s \{T_+\}$.
 \end{enumerate}
\end{lemma}

\begin{proof}
(i)
The first assertion is a consequence of
 \cite[Theorem 4.2.1]{DS10JAG} and Proposition \ref{p:holder}(i).

\medskip

(ii) 
As $f$ has simple action on cohomology, there exists a constant $s$
such that
$d_p^{-n}  (f^n)^* \{S\}\to s \{T_+\}$.
It is clear that the constant $s$ 
 depends linearly on $S$.
 When $S\geq 0$ the statement follows from (i). The general case follows by linearity.
\end{proof}

\medskip

As the inverse $f^{-1}$
of any automorphism with simple action on cohomology satisfies the same property,
the above
 also holds for the automorphism $f^{-1}$.
Since $(f^{-1})^* = f_*$, by Poincaré duality, we have $d_q (f)= d_{k-q}(f^{-1})$ for all $0\leq q\leq k$. Hence,
 the main dynamical degree of $f^{-1}$ is 
 equal to the one of $f$ and is the
 dynamical degree 
 of order $k-p$
 of $f^{-1}$.
The eigenspace associated to this eigenvalue for the action of $f_*$ on $H^{k-p,k-p} (X,\R)$
is a real line $H^-$. A \emph{Green $(k-p,k-p)$-current} $T_-$ of $f$
 is a non-zero positive closed $(k-p,k-p)$-current
 invariant under  $d_p^{-1}f_*$, i.e., satisfying $f_*  (T_-) = d_p T_-$.
 The cohomology class
  $\{T_-\}$ generates  $H^-$,
 $T_-$  is the unique positive closed $(k-p,k-p)$-current
 in $\{T_-\}$, and we have 
 $d_p^{-n} (f^n)_* S\to sT_-$
 for every $S\in \mathscr D_{k-p}$, where the constant $s$ depends linearly on $\{S\}$.

\medskip

Note that the currents $T_+$ and $T_-$ are unique up to multiplicative constants. We choose $T_+$ and $T_-$
such that the positive measure $T_+\wedge T_-$ is of mass 1. This is then the unique measure of maximal entropy
of $f$, see \cite{DS10JAG} for details.

\medskip

  The following result gives a quantitative
  description of the   convergences above, 
  see \cite[Proposition 3.1]{DS10}.
  Note that the independence of the constant
  $c$ from $S$ is given in \cite[Proposition 2.1]{DS10}.

\begin{proposition}\label{prop:cm}
Let $f$, $d_p$, $\delta'$
 be as above
 and $S$ be a current in $\mathscr D_{p}$ with $\|S\|_*\leq 1$.
Let $s$ be the constant such that
$d^{-n}_p (f^n)^* (S)$ converge to $sT_+$.
 Let $R$ be a current
in $\mathscr D^0_{k-p+1}$
with 
$\|R\|_* \leq 1$ 
and whose
$\alpha_{k-p+1}$-normalized
super-potential $\mathcal U_R$ 
is $(2,\lambda, 1)$-H\"older continuous for some $\lambda >0$.
Let 
$\mathcal U_{T_+}$ and
 $\mathcal U_n$ 
be the $\alpha_{p}$-normalized super-potentials of $T_+$
and $d^{-n}_p (f^n)^*(S)$, respectively. 
Then
\[|\mathcal U_n(R) - s\mathcal U_{T_+} (R)| \leq c (d_p/\delta')^{-n},\]
where $c > 0$ 
is a constant independent of $R$, $S$, $s$, and 
$n$.
\end{proposition}

In particular, we will need the following consequence of the above result, 
see also Taflin \cite[Theorem 3.7.1]{Taflin11}
for a similar result in the case where $p=1$.

\begin{corollary}\label{c:cm}
Let $f$, $d_p$, $\delta'$
 be as above
 and $S$ be a 
 current in $\mathscr D_{p}$
 with $\|S\|_*\leq 1$ and
 such that 
$d^{-n}_p (f^n)^* (S)\to 0$.
 Let $\xi$ be a $(k-p,k-p)$-current
with 
$\|dd^c \xi\|_* \leq 1$ 
and such that the
  super-potential $\mathcal U_{dd^c \xi}$ of $dd^c \xi$
  (which is independent of the normalization)
is $(2,\lambda, 1)$-H\"older continuous for some $\lambda >0$.
Assume that either $S$ or $\xi$ is smooth.
Then
\[|
\langle d_p^{-n} (f^n)^* (S), \xi \rangle | \leq c (d_p/\delta')^{-n},\]
where $c > 0$ 
is a constant independent of $R$, $S$, and 
 $n$.
\end{corollary}

Observe that, since either $S$ or $\xi$ is smooth, the pairing in the statement is well defined.

\begin{proof}
We first consider the case where $S$ is exact. Recall that, in this case, the super-potential
of $S$ is independent of the normalization. As
we have $\{(f^n)^* (S)\}= (f^{n})^* \{S\}=0$,
the same is true for the super-potential
 of $(f^{n})^* (S)$ for all $n\in \mathbb N$. Setting $S_n := d_p^{-n} (f^n)^* (S)$, we then have
 \[
|\langle S_n, \xi\rangle|
=
|\mathcal U_{S_n} (dd^c \xi)|.
 \]
 By the assumptions on $\xi$, we can apply Proposition \ref{prop:cm} with $dd^c \xi$ instead of $R$ and $s=0$.
 The assertion in this case follows.

\medskip

Let us now consider the general case.
Observe that 
$d_p^{-n} (f^n)^*\{S\}\to 0$. So,
the set of 
the classes $\{S\}$
of the currents $S \in \mathscr D_p$ 
with this property is an
 hyperplane $H\subset H^{p,p} (X,\R)$,
which is a complement of the line generated by $\{T_+\}$ and is invariant under the action of $f^*$.
For simplicity,
denoting by $h$ the dimension of $H^{p,p}(X,\R)$,
we let $\alpha_1, \dots, \alpha_{h-1}$ be real smooth $(p,p)$-forms 
such that $\{\alpha_1\}, \dots, \{\alpha_{h-1}\}$ form a basis for $H$, and $\alpha_h$ be a smooth form in the class 
of $T_+$. We will consider 
 $\alpha$-normalized super-potentials of currents in $\mathscr D_p$, with $\alpha = (\alpha_1, \dots, \alpha_h)$.
 We also fix 
 $\check\alpha = (\check\alpha_1, \dots, \check\alpha_h)$
 a dual basis of $\alpha$.
 As $\{S\}\in H$, we have
 \[
S = a_1 \alpha_1 + \dots + a_{h-1} \alpha_{h-1} +
 S''
 \]
for some $S'' \in \mathscr D_{p}^0$
and $a_j \in \mathbb R$.
By the first part of the proof applied to $S''$,
it is enough to prove the statement for 
$S' := \sum_{j=1}^{h-1} a_j \alpha_j$ 
instead of $S$. Observe, in particular, that $S'$ is smooth and that 
the $\alpha$-normalized super-potential $\mathcal U_{S'}$ of $S'$
satisfies
$\mathcal U_{S'}=0$ by Lemma \ref{l:super-1}(iii).

\medskip

Denote $S'_n := d_p^{-n} (f^n)^*(S')$. We have
\[
S'_n = 
\sum_{j=1}^{h} c_{n,j} \alpha_j + dd^c V_n,
\]
where
the $c_{n,j}$
are defined by 
$\{S'_n\} = \sum_{j=1}^h c_{n,j} \{\alpha_j\}$
and
the $(p-1,p-1)$-current $V_n$  is chosen so
 that $\langle V_n, \check\alpha_j \rangle=0$ for all $1\leq j\leq h$. Observe that
$c_{n,h}=0$ for all $n$ 
 because of the invariance of $H$. It follows that, for any $\xi$ as in the statement, we have
\[
\langle S'_n , \xi\rangle
=
\sum_{j=1}^{h-1} c_{n,j} \langle \alpha_j, \xi\rangle +
\langle dd^c V_n, \xi\rangle=
 \sum_{j=1}^{h-1} c_{n,j} \langle \alpha_j, \xi\rangle +
 \mathcal U_{S'_n} ( dd^c \xi),
\]
where $\mathcal U_{S'_n}$ is the $\alpha$-normalized super-potential of $S'_n$.
By the assumptions on $\xi$ and Proposition \ref{prop:cm} applied with $dd^c \xi$
instead of $R$ and with $s=0$, 
we have
$|\mathcal U_{S'_n} ( dd^c \xi)|\lesssim (d_p/\delta')^{-n}$.
Hence, 
it is enough to prove that, for all $1\leq j \leq h-1$, we have
$|c_{n,j}|\lesssim (d_p/\delta')^{-n}$.

\medskip

Let $A$ be the $(h-1)\times (h-1)$ matrix
representing
$f^{*}_{|H}$ with respect to the basis $\{\alpha_j\}_{1\leq j \leq h-1}$, i.e., whose $j$-th column is given 
by the coordinates of $f^*\{\alpha_j\}$
with respect to the given basis.
Denoting $\underline c_n := (c_{n,1}, \dots, c_{n,h-1})^t$ and $\underline a := (a_{1}, \dots, a_{h-1})^t$, we see that
\[
\underline c_n = d_p^{-n}  A^n \underline a.
\]
 As the spectral radius of the action
 of  $f^*$ on $H$ is smaller than $\delta'$, we see that $\|A^n\| = o (\delta'^n)$. It follows that
 $|c_{n,j}|\lesssim (\delta'/d_p)^{-n}$, as desired. 
 The assertion follows.
\end{proof}

In our study we will also need the case where $q=p+1$.
It follows from the definition of $d_{p+1}$ and 
the fact that $d_{p+1}<d_p$
that
 $d_p^{-n}
(f^{n})^* (S)\to 0$ as
$n\to \infty$
for every $S\in \mathscr D_{p+1}$.
We will need later
a more quantitative version of this convergence, for $S$ with H{\"o}lder
continuous super-potentials.
 This is given by the next proposition. Recall that we normalize potentials with respect to a given
 $\alpha_{p+1} = (\alpha_{p+1,1},\dots, \alpha_{p+1,h_{p+1}})$,
 where $h_{p+1}$ is the dimension of $H^{p+1,p+1} (X,\R)$.

\begin{proposition}\label{p:sequence-holder}
Let $f$ be as above.
 Take $S \in
 \mathscr D_{p+1}$
with $\|S\|_*\leq 1$.
\begin{enumerate}
\item[{\rm (i)}] $\| d_p^{-n} (f^n)^* S\|_*\to 0$ as $n\to \infty$; in particular, $d_p^{-n} (f^n)^* S\to 0$ as $n\to \infty$.
\item[{\rm (ii)}] 
Assume that $S$ has an $(l,\lambda, M)$-H{\"o}lder continuous $\alpha_{p+1}$-normalized
super-potential.
Then there exist $\lambda'$ and $M'$
 independent of $S$ 
such that,
for every $n\geq 0$,
$d_p^{-n} (f^n)^* S$ has  a $(2,\lambda', M')$-H{\"o}lder continuous 
$\alpha_{p+1}$-normalized
super-potential.
\end{enumerate}
\end{proposition}

\begin{proof}
The first assertion is
 a consequence of the inequality $d_{p+1} < d_p$
 and of the fact that, for every $\epsilon>0$, we have
  $\|(f^n)^* (S)\|_*\leq C (d_{p+1}+\epsilon)^n \|S\|_*$,
  where the constant $C$ is independent of $S$ because the mass of a positive closed current only depends on its cohomology class.
It remains to prove the second assertion.
 By Proposition \ref{p:holder}(i)
 we can assume that $l=4$.
 
 \medskip
 
\noindent
 {\bf Particular case.}
 \smallskip
We first prove the third assertion 
assuming that
 $S$ is exact.
Recall that, in this case, the super-potentials $\mathcal U_S$ is independent of the normalization, see Lemma \ref{l:super-1}(ii).

\medskip

Fix $R\in 
\mathscr D^0_{k-p}$ with $\|R\|_*\leq 1$
and set $S_n := 
d_p^{-n}(f^n)^* S$.
We need
 to show that
\[
|\mathcal U_{S_n} (R)
| \leq M'
\|R\|_{\Cc^{-4}}^{\lambda'}
\]
for some $\lambda'$ and $M'$
independent of $S$ and $R$.
 The assertion is then a consequence of Proposition \ref{p:holder}(i).

 \medskip
 
 Fix $L\geq \max (2, \sup_{\|\Phi\|_{\Cc^4}\leq 1} \|f^* \Phi\|_{\Cc^4})$
 where the supremum
 is on smooth $(q-1,q-1)$-forms $\Phi$ on $X$.
As $S$ and $R$
are both
exact, by 
Lemma \ref{l:super-1}(ii)
and the definition of super-potential
 we have
\begin{equation}\label{eq:USnR}
|\mathcal U_{S_n} (R)
|
= 
d_p^{-n}
\big| \mathcal U_{S} \big((f^n)_* ( R)
\big)
\big|
\leq 
M
 \|
 d_p^{-n}
  (f^n)_* (R)\|_{\Cc^{-4}}^{\lambda}.
  \end{equation}
By Poincaré duality,
 $\|
 d_p^{-n} (f^n)_* (R)\|_*$ is bounded independently of $n$.
 We also have
\[
\begin{aligned}
\|
d_p^{-n}
(f^n)_* (R)\|_{\Cc^{-4}}
& 
=
\sup_{\|\Phi\|_{\Cc^{4}}\leq 1}
\big|
\big\langle
d_p^{-n} 
(f^n)_* (R), 
\Phi
\big\rangle
\big|
\\
&
=\sup_{\|\Phi\|_{\Cc^{4}}\leq 1}
\big|
\mathcal U_{d_p^{-n} (f^n)_* (R)} (dd^c \Phi)
\big|\\
& \lesssim \theta^n
\end{aligned}\]
for some $0<\theta<1$, where the last step follows
from Corollary \ref{c:cm} applied with $f^{-1}$, $R$, and $\Phi$
instead of $f$, $S$, and $\xi$. Observe that the assumption on $dd^c \xi$
 in that corollary
is satisfied by  $dd^c \Phi$ by
Lemma \ref{l:holder}(i).

\medskip

Assume first that $n\geq -(2\log L)^{-1} \log \|R\|_{\Cc^{-4}}$.
In this case, we have
\[
\|
d_p^{-n}
(f^n)_* (R)\|_{\Cc^{-4}}
\lesssim \theta^n
\leq
\|R\|_{\Cc^{-4}}^{\frac{1}{2}\frac{|\log \theta|}{\log L}}.
\]
Together with \eqref{eq:USnR}, this implies the assumption.

\medskip

Assume instead that $n\leq -(2\log L)^{-1} \log \|R\|_{\Cc^{-4}}$. 
Observe that this implies that $L^n \|R\|_{\Cc^{-4}}\leq \|R\|_{\Cc^{-4}}^{1/2}$.
Hence, for  all such $n$, we have
\[
\|
d^{-n}_p
(f^n)_* R\|_{\Cc^{-4}}\leq 
\sup_{\|\Phi\|_{\Cc^4}\leq 1}
d_p^{-n}
|\langle R, (f^n)^* \Phi\rangle|\leq 
d_p^{-n} L^n \|R\|_{\Cc^{-4}}\leq d_p^{-n} \|R\|_{\Cc^{-4}}^{1/2},
\]
which also implies the assertion in this case.
The proof in the particular case
is complete.

\medskip

\noindent
{\bf General case.}
We now consider the general case.
By the previous part of the proof, and
as $H^{p+1,p+1} (X,\R)$
is finite dimensional, 
it is enough to prove the assertion 
 for any
 finite family
 of 
smooth forms whose classes generate
$H^{p+1,p+1} (X,\R)$. Therefore, it is enough to prove the statement for a fixed smooth form
$S$.

\medskip
 Let $1\leq m< \dim H^{p+1,p+1}(X,\R)$
 be the minimal integer such that $\{S\}$, $\{f^{*} (S)\}, \dots, \{(f^{m})^* (S)\}$ 
 are linearly
  dependent over $\R$,
  and define $a_0, \dots, a_{m-1}$
 by the relation
 \begin{equation}\label{e:basis}
\{(f^{m})^{*} (S)\} = a_0 \{S\} + \dots + a_{m-1} \{ (f^{m-1})^* (S)\}.  
 \end{equation}
 Let $E$ be the subspace of $H^{p+1,p+1}(X,\R)$
spanned
by $\{S\}, \dots, \{ (f^{m-1})^* (S)\}$. Then these $m$
 classes form a basis $\mathcal B$ of $E$. 
The action of
$f^*_{|E}$ 
with respect to the basis
 $\mathcal B$ is represented by the  $m\times m$ matrix
 \[
A_E:=
\left(
\begin{array}{ccccc}
0 & 0 & \dots & 0 & a_0\\
1 & 0 & \dots & 0 & a_1\\
0 & 1 & \dots & 0 & a_2\\
\vdots & \vdots & \ddots & \vdots & \vdots\\
 0 & 0 & \dots & 1 & a_{m-1}
\end{array} 
\right).
 \]
 We denote by $B_E$ the transpose of $A_E$.
Fix $0<\epsilon< d_p -d_{p+1}$.
By the definition of $d_{p+1}$, 
we have $\|A_{E}^n\| = \|B_E^n\|\lesssim (d_{p+1}+\epsilon)^n = o(d_p^{n})$.
  
  \medskip

  It follows from \eqref{e:basis} that 
  \[
U:= (f^{m})^{*} (S) -\sum_{j=0}^{m-1}  a_{j} (f^{j})^* (S) 
  \]
is an exact
form, i.e., it belongs to $\mathscr D^0_{p+1}$. 
Since $(f^j)^* (S)$,
$0 \leq j \leq m-1$, and
$U$ are smooth, they have
 $(l, \lambda,N)$-H{\"o}lder continuous 
 $\alpha_{p+1}$-normalized
 super-potentials
 for some constant $N$.
  
  \medskip
  
  For every $n\geq 1$, set
  \[
\mathbb W_n :=
\left(
\begin{array}{c}
(f^n)^* (S)\\
(f^{n+1})^* (S)\\
\vdots\\
(f^{n+m-1})^* (S)
\end{array}
\right)  
\quad
\mbox{ and }
\quad
\mathbb U_n :=
\left(
\begin{array}{c}
0\\
0\\
\vdots\\
(f^n)^* (U)
\end{array}
\right),
  \]
  where both $\mathbb W_n$ and $\mathbb U_n$ have $m$ components.
As $\mathbb W_1 = B_{E} \mathbb W_0 + \mathbb U_0$, we obtain by induction that
\[
d_p^{-n}
\mathbb W_n = 
d_p^{-n}
B_{E}^{n} \mathbb W_0 +
d_p^{-n}
 \sum_{j=0}^{n-1}B_E^{n-j} \mathbb U_j.
\]
As the first component
of $\mathbb W_n$
is equal to
$(f^n)^* (S)$, we need to consider
the $\alpha_{p+1}$-normalized
super-potential
 of the first component of the RHS of the above expression.

\medskip

By the above,
 the currents 
$(f^{j})^* (S)$, $0\leq j \leq m-1$,
 have
an $(l,\lambda, N)$-H{\"o}lder continuous 
$\alpha_{p+1}$-normalized
super-potential.
As
$\|B_E^n\|= o( d_{p+1}^n)$,
 it follows that the
$\alpha_{p+1}$-normalized
super-potential of
every
 component
of 
$d_{p}^{-n}
B_{E}^{n} \mathbb W_0$ is
$(l, \lambda, M'/2)$-H{\"o}lder continuous, for some
$M'$ large enough.

\medskip

In order to prove the assertion, 
using again that $\|B_E^{n-j}\|= o(d_p^{n-j})$,
 it is enough to show that 
 there exist $l',\lambda',M'$
such that,
for every $n\geq 0$,
$U_n := d_p^{-n} (f^n)^* (U)$
has  a  $(l',\lambda', M'/2)$-H{\"o}lder
continuous 
$\alpha_{p+1}$-normalized
super-potential.
As $U$ is exact, this follows from the 
particular case considered above.
This concludes the proof.
\end{proof}

\subsection{Further properties of the Green currents}
We prove here two lemmas that we will need in the next section in order to apply Proposition \ref{prop:cm}
and Corollary \ref{c:cm}.
As in the previous section,
we let $f$ be an automorphism of $X$ with simple action on cohomology, and $d_p$
 its largest dynamical degree. In particular, the Green current $T_+$ has bi-degree $(p,p)$.
 We also fix a normalization $\alpha_{p+1}:=(\alpha_{p+1,1}, \dots, \alpha_{p+1, h_{p+1}})$, where
 $h_{p+1}$ is the dimension of $H^{p+1,p+1}(X,\R)$.

\begin{lemma}\label{l:new-norm-star}
Let $f$ be as above.
Let  $\kappa \geq 1$ be an integer and
 $g_0, \ldots,g_\kappa\colon X\to \R$
satisfy $\|g_j\|_{\Cc^2}\leq 1$. 
 Then there is a positive constant $c$ independent of the $g_j$'s
  such that for all $\ell_0, \ldots, \ell_{\kappa}\geq 0$ we have
\[\|dd^c \big((g_0\circ f^{\ell_0})\dots (g_\kappa\circ f^{\ell_\kappa})\big)\wedge {T}_+\|_{*}
\leq c.\]
\end{lemma}

\begin{proof}
Set $\tilde g_j := g_j \circ f^{\ell_j}$ and notice
that
$ i \partial \tilde g_j \wedge \bar \partial \tilde g_j = (f^{\ell_j})^* (i\partial g_j \wedge \bar \partial g_j)$
for all $j$. 
We have
\[\begin{aligned}
dd^c \big(\tilde g_0 \dots \tilde g_\kappa\big) = \sum_{j=0}^\kappa dd^c \tilde g_j \prod_{l\neq j} \tilde g_{l}
+
\sum_{0\leq j\neq  l \leq \kappa}
 i \partial \tilde g_j \wedge \bar \partial \tilde g_l 
\prod_{m\neq j,l} \tilde g_m.
\end{aligned}
\]
Since
 $\|g_j\|_{\Cc^2}\leq 1$ 
we have
 $|g_j|\leq 1$
 and $|dd^c g_j|\lesssim \omega$. 
Then
\[
\Big|
\sum_{j=0}^\kappa dd^c \tilde g_j \prod_{l\neq j} \tilde g_{l}
\Big| \ \lesssim \ \sum_{j=0}^\kappa (f^{\ell_j} )^* (\omega)
\]
and an application of Cauchy-Schwarz inequality gives
\[
\Big| 
\sum_{0\leq j\neq l\leq \kappa}
i \partial \tilde
g_j 
\wedge \overline \partial 
\tilde g_l
\prod_{m\neq j,l} 
\tilde g_m
\Big|
\lesssim 
 \sum_{j=0}^\kappa
i \partial \tilde g_j \wedge \overline \partial \tilde g_j 
\lesssim \sum_{j=0}^\kappa
(f^{\ell_j})^* (\omega).
\]

 We deduce from the above inequalities
 and the invariance of $T_+$
  that
\begin{equation}\label{eq:sum-omega-T}
\begin{aligned}
\big|dd^c \big((g_0\circ f^{\ell_0})\ldots (g_\kappa\circ f^{\ell_\kappa})\big)\wedge {T}_+\big|
& \ \lesssim \ 
\sum_{j=0}^\kappa 
\big(f^{\ell_j} )^* (\omega \big)\wedge  T_+\\
& \ = \ 
\sum_{j=0}^\kappa
(f^{\ell_j} )^* ( \omega) \wedge   d_p^{-\ell_j} (f^{\ell_j} )^* ( T_+)\\
& \ = \ 
\sum_{j=0}^\kappa
d_p^{-\ell_j}
 (f^{\ell_j} )^* \big( \omega \wedge  T_+\big).
\end{aligned}
\end{equation}

We now
 use that the $(p+1,p+1)$-current $\omega\wedge T_+$ is positive
and closed.
We have
 $$\|(f^{\ell_j} )^* \big( \omega \wedge T_+\big)\| = \big \langle (f^{\ell_j} )^* \big( \omega \wedge  T_+\big), \omega^{k-p-1} \big \rangle= \big \langle \omega \wedge  T_+, (f^{\ell_j})_*(\omega^{k-p-1}) \big\rangle,$$
where the last form is positive closed.
The last pairing only depends on the cohomology classes of $\omega$, $T_+$,
and $(f^{\ell_j})_* (\omega^{k-p-1})$.
Hence, it is 
bounded by a constant times
 the mass of 
$(f^{\ell_j})_*(\omega^{k-p-1})$.
Since
\[
\|
(f^{\ell_j})_*(\omega^{k-p-1})\|=
\langle \omega^{p+1}, (f^{\ell_j})_*(\omega^{k-p-1})\rangle
=
\langle (f^{\ell_j})^* (\omega^{p+1}), \omega^{k-p-1}\rangle,
\]
for all $\epsilon >0$
this number 
 is also equal to
 $\|(f^{\ell_j})^* (\omega^{p+1})\|\lesssim (d_{p+1}+\epsilon)^n$.
As $d_{p+1}<d_p$
by assumption, it follows that 
each term in 
 the last sum in \eqref{eq:sum-omega-T} is bounded,
 which implies that the sum is also
 bounded.
The lemma follows.
\end{proof}

\begin{lemma}\label{l:R-holder}
Let $f$ be as above.
Let  $\kappa \geq 1$ be an integer and
 $g_0, \ldots,g_\kappa\colon X\to \R$
satisfy $\|g_j\|_{\Cc^2}\leq 1$. 
 Then there exist constants
 $0<\lambda\leq 1$ and $M>0$, independent of the $g_j$'s,
  such that for all $\ell_0, \ldots, \ell_{\kappa}\geq 0$ the current
\[ dd^c \big( (g_0\circ f^{\ell_0})\dots (g_\kappa\circ f^{\ell_\kappa}) \big)\wedge {T}_+
\]
has a 
$(2,\lambda,M)$-H{\"o}lder continuous
 super-potential.
\end{lemma}

Observe that the current in the statement is exact, hence its super-potential is independent of the normalization.

\begin{proof}
We denote the current in the statement by $R$. It is a 
real exact $(p+1,p+1)$-current.
Recall that, by 
Proposition \ref{p:holder}(i), it is enough to show that $R$
has an  $(l,\lambda,M)$-H{\"o}lder continuous 
super-potential, for some
$l>0$, $0<\lambda\leq 1$, and $M>0$
independent of $g_0, \dots, g_\kappa$ and $\ell_0, \dots, \ell_\kappa$.

\medskip

By 
Proposition \ref{p:holder}(ii),
it is enough to show that  there exists a positive closed current $V$ with
a $(2,\lambda', M')$-H{\"o}lder
continuous 
$\alpha_{p+1}$-normalized
super-potential
such that $|R| \leq V$.
Indeed, this implies that both $V+R$ and $V-R$
are positive, closed, and bounded by $2V$. Hence, they have $(2,\lambda'',M'')$-H{\"o}lder continuous 
$\alpha_{p+1}$-normalized
super-potentials
for some constants
$0\leq \lambda''\leq 1$ and $M'' >0$.
 It follows that
$R$ has  a $(2,\lambda'',M'+M'')$-H{\"o}lder continuous 
super-potential.

\medskip

 We have already seen
 in \eqref{eq:sum-omega-T} that
\[
|dd^c \big(g_0\circ f^{\ell_0})\dots (g_\kappa\circ f^{\ell_\kappa}) \big) \wedge {T}_+| \lesssim
\sum_{j=0}^\kappa 
d_p^{-\ell_j}
 (f^{\ell_j} )^* \big( \omega \wedge T_+\big).
\]
In order to prove the assertion, it is then enough to prove that there exist 
$\lambda'$ and $M'$ such that, for
every $j\geq 0$,
the current $d_p^{-j}
 (f^{j} )^* \big( \omega \wedge  T_+\big)$ has
 a $(2,\lambda', M')$-H{\"o}lder continuous
 $\alpha_{p+1}$-normalized
  super-potential.
 For $j=0$, this follows from Lemma \ref{l:T+}(i)
 and Proposition \ref{p:holder}(iii).
 Proposition \ref{p:sequence-holder}(ii)
 implies that that same holds for all $j \in\mathbb N$.
The proof is complete.
\end{proof}

\section{Exponential mixing of all orders and Central Limit Theorem}

We prove now Theorem \ref{t:intro-mixing}. By interpolation techniques
\cite{Triebel95}
 as in \cite[pp. 262-263]{Dinh05}, it 
is enough to prove the statement in the case $\gamma=2$. 
The statement is clear for $\kappa=0$. By induction, we can assume that
the statement holds for up to $\kappa$
functions and prove it for $\kappa+1\geq 2$ functions.
In particular, by the induction assumption, 
we are allowed to add to the $g_j$'s
 a constant
and assume that $\langle\mu, g_j\rangle=0$
for all $0\leq j\leq \kappa$. 
We can also assume that 
$\|g_j\|_{\Cc^2}\leq 1$ for all $0\leq j \leq \kappa$.
Then, we need to show that
\[\Big|
\langle \mu,
 g_0 ( g_1 \circ  f^{n_1} )
\dots
 (g_{\kappa} \circ f^{n_{\kappa}})
\rangle
\Big|
\lesssim 
d^{-\min_{0\leq j \leq \kappa-1} ( n_{j+1}- n_j)/2}.\]
We can also assume that $n_1$ is even.
Indeed, by the invariance of $\mu$,
the case where $n_1$
is odd 
can be treated similarly by replacing 
$n_j$ with $n_j-1$ for $1\leq j \leq \kappa$ and
$g_0$ with $g_0 \circ f^{-1}$.

\medskip

Consider the automorphism 
$F$ of the compact K{\"a}hler
manifold
$\mathbb X := X\times X$ given by $F(z,w):= (f(z),f^{-1}(w))$. 
We first show that
$F^{-1}=(f^{-1},f)$ 
satisfies the assumptions of Proposition \ref{prop:cm}
and Corollary \ref{c:cm}.
Recall that $\delta$ is the 
 auxiliary degree of $f$ and that by assumption we have $\delta<d_p$.

\medskip

\begin{lemma}
\label{l:F-applies}
The automorphisms $F$ and $F^{-1}$ 
of $\mathbb X$
have simple action on cohomology.
  More precisely, 
the dynamical degrees of order $k$ of $F$ 
and $F^{-1}$
are equal to $d_p(f)^2$, 
they are eigenvalues 
of multiplicity one
of both $F^*$ and $F_*$
acting on $H^{k,k} (\mathbb X,\C)$, 
and
 all the other dynamical degrees of $F$ and $F^{-1}$,
  as well
as the other eigenvalues of $F^*$ and $F_*$
 on $H^{k,k} (\mathbb X,\C)$, 
have modulus
smaller than or equal to $d_p \cdot \delta$.
\end{lemma}

\begin{proof}
By K\"unneth formula \cite[Theorem 11.38]{Voisin02}, for every $0\leq q \leq k$
we have a canonical isomorphism
\[
H^{q,q} (X\times X, \C) = \bigoplus_{s+r=q} H^{s,r} (X,\C)\otimes H^{r,s} (X,\C).
\]
The operators $F^*$ and $F_*$
preserve
 the above
decomposition. By \cite{Dinh05JGA}, the spectral radius  of $f^*$ on
$H^{r,s}(X,\C)$ is (equal to the spectral radius of $f_*$ on $H^{k-r,k-s}(X,\C)$ and)
smaller than or equal to $\sqrt{d_r d_s}$. 
The assertion follows from the
fact that $f$ 
and $f^{-1}$ have
 simple actions on cohomology and the  definitions of $p$ and $\delta$.
\end{proof}

 In  this section,
 we will denote
 by  $T_{\pm}$
 the Green currents of $f$
  and 
  by $\mathbb T_{\pm}$
the
 Green currents of $F$.
More precisely, we fix  Green currents $T_{\pm}$
for $f$
(they are unique up to a multiplicative constant)
 such that 
$\mu= T_+\wedge T_-$ and we set
\[
\mathbb T_+ = T_ + \otimes T_- \quad \mbox{ and }
\quad 
\mathbb T_- = T_- \otimes T_+,  
\] 
see \cite[Section 4.1.8]{Federer} 
for the tensor (or cartesian) product of currents.
Since $f^* (T_+) = d_p T_+$
and
$f_* (T_-)=d_p T_-$,
 we have $F^* ( \mathbb T_+)= d^2_p \mathbb T_+$
 and $F_* (\mathbb T_-)= d^2_p \mathbb T_-$.

\medskip

Set $h:= g_1 (g_2 \circ f^{n_2-n_1}) \ldots (g_{\kappa}\circ f^{n_\kappa-n_1})$. 
Recalling that $\langle\mu, g_0\rangle=0$ and that $\|g_j\|_{\Cc^2}\leq 1$,
by the induction hypothesis
it is enough to show that
\[
\big|
\langle 
\mu, g_0 (h\circ f^{n_1})\rangle
\big|
\lesssim
 d_p^{-n_1/2}.
\]

We denote by $(z,w)$ the coordinates on $\mathbb X= X\times X$ and we set
\[\Psi (z,w):= g_0(w) h(z).\]

\begin{lemma}\label{l:properties-R}
The following assertions hold:
\begin{enumerate}
\item[{\rm (i)}] $\|dd^c \Psi \wedge \mathbb T_+\|_{*} \leq c$;
\item[{\rm (ii)}] $dd^c \Psi  \wedge \mathbb T_+ $ has 
 a $(2,\lambda, M)$-H\"older continuous
super-potential,
\end{enumerate}
where $c$, $\lambda$, and $M$
 are positive constants depending on $\kappa$, but not on the $g_j$'s and the $n_j$'s.
\end{lemma}

\begin{proof}
The two assertions are consequences of Lemmas
\ref{l:new-norm-star} and \ref{l:R-holder}
applied to $F\colon \mathbb X \to \mathbb X$,
 respectively.
 Here, we use the functions 
 $\tilde g_0 (z,w):= g_0 (w)$
and $\tilde g_j (z,w):= g_j (z)$ for $j\geq 1$, and the integers
$\ell_0:=0$,  $\ell_1:=0$, and $\ell_j := n_{j+1}-n_j$
for $j\geq 2$.
\end{proof}

\begin{proof}[End of the proof of Theorem \ref{t:intro-mixing}]
Recall that we are assuming that $n_1$ is even. By the invariance of $\mu$, we see that
\[
\langle 
\mu, g_0 (h\circ f^{n_1})\rangle
=
\langle
\mu, (g_0 \circ f^{-n_1/2})( h \circ f^{n_1/2})
\rangle.
\]
We first transform the above
integral on $X$
to an integral on $\mathbb X$
by means of the map $F$. Namely, 
using the invariance of $\mathbb T_+$,
we have
\begin{equation}\label{e:for-mixing}
\begin{aligned}
\big\langle
\mu, 
(g_0 \circ f^{-n_1/2})(h \circ f^{n_1/2})
\big\rangle
&=
\big\langle 
T_+ \wedge T_-,  (g_0 \circ f^{-n_1/2} )(h \circ f^{n_1/2})\big\rangle\\
& =
\big\langle
(T_+ \otimes  T_-) \wedge [\Delta],
(g_0 \circ f^{-n_1/2} (w) )(
 h\circ f^{n_1 /2} (z))\big\rangle\\
 &=
 \big\langle
d_p^{-n_1} (F^{n_1/2})^* (\mathbb T_+  )\wedge [\Delta],
(F^{n_1/2})^*
 \Psi \big\rangle
 \\
 &=
 \big\langle
\mathbb T_+\wedge
(d_p^{-n_1}  (F^{n_1/2})_* [\Delta]),
\Psi 
  \big\rangle\\
 & =
  \big\langle
 d_p^{-n_1} (F^{n_1/2})_* [\Delta],
 \Psi 
   \mathbb T_+ \big\rangle.
 \end{aligned}
\end{equation}
Observe, in particular, that the wedge product $[\Delta]\wedge \mathbb T_+$ is well defined by Proposition \ref{p:super-pot-1}(iv).
\begin{lemma}\label{l:Delta-s1}
We have $d_p^{-2n} (F^{n})_* [\Delta]\to \mathbb T_-$
as $n\to \infty$.
\end{lemma}

\begin{proof}
By Lemma \ref{l:F-applies}, $F^{-1}$ has simple action on
 cohomology and its main dynamical degree is that
of order $k$, which is equal to $d_p^2$.
As $[\Delta]$ has bi-degree $(k,k)$,  
there exists $s \in \mathbb R$ such that
$d_p^{-2n} (F^{n})_* [\Delta]\to s\mathbb T_-$ as $n\to \infty$. Hence, we only need to show that $s=1$.
We have 
\[
\big\langle d_p^{-2n} (F^n)_* [\Delta], \mathbb T_+ \big\rangle\to 
\langle s\mathbb T_-, \mathbb T_+ \rangle 
= s\langle T_-\otimes T_+, T_+\otimes T_-\rangle
=
s\|\mu\|^2=s.
\]
On the other hand, by the invariance of $\mathbb T_+$,
for every $n\in \mathbb N$ we also have
\[
\big\langle d_p^{-2n} (F^n)_* [\Delta], \mathbb T_+ \big\rangle
=
\big\langle  [\Delta],  d_p^{-2n} (F^n)^*\mathbb T_+ \big\rangle=
\langle  [\Delta],T_+ \otimes T_-\rangle = \|\mu\|=1.
\]
The assertion follows.
\end{proof}

We will apply the results of Sections \ref{s:super-potentials}
and \ref{s:green}
 to the
automorphism
$F^{-1}$ of $\mathbb X$.
In order to do this, 
for simplicity,
we let $h$ be the dimension of $H^{k,k} (\mathbb X, \R)$
and we fix a collection
 $\alpha := (\alpha_1, \dots, \alpha_h)$  of real smooth $(k,k)$-forms
on $\mathbb X$ 
with the property 
that  $\{\alpha_1\}= \{\mathbb T_-\}$ 
and $\{\alpha_2\}, \dots, \{\alpha_h\}$ are a basis for the
 $F_*$-invariant hyperplane in $H^{k,k} (\mathbb X,\R)$
transversal to the line generated
by $\{T_-\}$.
In the following, we only consider $\alpha$-normalized super-potentials
for currents in $\mathscr D_k (\mathbb X)$.

\medskip

By Lemma \ref{l:properties-R},
the current $R:= dd^c \Psi \wedge \mathbb T_+$
satisfies $\|R\|_*\lesssim 1$ and
has a  
$(2,\lambda,M)$-H{\"o}lder continuous potential
$\mathcal U_R$, for some $0<\lambda\leq 1$
 and $M>0$ independent of the
 $g_j$'s and $n_j$'s.
 By linearity, up to multiplying $g_0$ by
  a constant, we can assume that $\|R\|_*\leq 1$ and 
  that $\mathcal U_R$ is $(2,\lambda,1)$-H{\"o}lder continuous, for some $0<\lambda\leq 1$.
Recall that
 $\delta  <\delta' < d_p$. 
 By Lemmas \ref{l:F-applies}
 and \ref{l:Delta-s1},
 we can apply 
Corollary \ref{c:cm}
with $F^{-1}$, 
$\Psi \mathbb T_+$, and  $n_1/2$
 instead of $f$,
 $\xi$, and $n$,
and get that, for all smooth $S\in \mathscr D_k (\mathbb X)$
such that $d_p^{-2n} (F^n)_* (S)\to 0$,
we have
\[
\big|
\big\langle d_p^{-n_1} (F^{n_1/2})_* (S), \Psi\mathbb T_+\big\rangle
\big|
\lesssim (d_p / \delta')^{-n_1/2}.
\]
Let now $S_\epsilon$ be a regularization of $[\Delta]-\mathbb T_-$, with $S_\epsilon \to [\Delta]-\mathbb T_-$ as $\epsilon\to 0$
and $\|S_\epsilon\|_*$ bounded \cite{DS04}.
We can also add to $S_{\epsilon}$ a small smooth form so that
$\{S_\epsilon \} = \{ [\Delta]-\mathbb T_-\}$.
We apply the above inequality with $S_\epsilon$ instead of $S$ and take $\epsilon\to 0$. It follows 
 from Proposition \ref{p:super-pot-1}(iv)
 and  Lemma \ref{l:properties-R}
that
\begin{equation}\label{e:from-prop:cm}
\big|
\big\langle d_p^{-n_1} (F^{n_1/2})_* ([\Delta]-\mathbb T_-), \Psi\mathbb T_+\big\rangle
\big|
\lesssim (d_p / \delta')^{-n_1/2}.
\end{equation}
Combining  \eqref{e:for-mixing}
and \eqref{e:from-prop:cm}
and using the invariance of $\mathbb T_-$
we obtain that
\[
\big|\big\langle
\mu, 
(g_0 \circ f^{-n_1/2})( h \circ f^{n_1/2})
\big\rangle
- 
\langle\mathbb T_-, \Psi \mathbb T_+\rangle \big| \lesssim (d_p/\delta')^{-n}.
\]
To conclude, it is enough to show that
$\langle\mathbb T_-, \Psi \mathbb T_+\rangle=
0$. As
$\langle \mu, g_0\rangle=0$,
we have
\[
\langle \mathbb T_-, \Psi \mathbb T_+ \rangle
=
\langle
T_-\otimes T_+, g_0(w) h (z) T_+\otimes T_-\rangle
=
\langle\mu, g_0\rangle \cdot \langle \mu, h \rangle=0.
\]
The proof is complete.
\end{proof}

\begin{remark}\label{r:optimal-rate}
The rate of mixing in Theorem \ref{t:intro-mixing}
can be improved by considering $F' := (f^l,f^{-m})$ 
instead of $F$, for suitable positive integers $l$ and $m$,
see for instance
 \cite[Remark 4.1]{DS10}.
\end{remark}


\begin{thebibliography}{99}


\bibitem{BK09}
Bedford E. and Kim K.,
Dynamics of Rational Surface Automorphisms: Linear Fractional Recurrences,
{\it J. Geom. Anal.} {\bf 19} (2009), no. 3, 553-583.




\bibitem{BD23}
Bianchi, F. and Dinh, T.-C., Every complex H{\'e}non maps is exponential mixing of all orders and satisfies the CLT, 
{\it preprint} (2023),
{\tt arXiv:2301.13535}.

\bibitem{BG}
Bj{\"o}rklund, M. and Gorodnik, A.,
Central Limit Theorems  for group actions  which are exponentially mixing of all orders
{\it Journal d'Analyse Math{\'e}matique} {\bf 141} (2020),
457-482.

\bibitem{Cantat01}
Cantat, S.,
Dynamique des automorphismes des surfaces $K3$,
{\it Acta Math.}
 {\bf 187} (2001), no. 1,
1-57.

\bibitem{CD20}
Cantat, S. and Dupont, C.,
Automorphisms of surfaces: Kummer rigidity and measure of maximal entropy,
{\it  J. Eur. Math. Soc. (JEMS)} {\bf 22} (2020), no. 4, 1289–1351. 


\bibitem{DH22}
Dang, N.-B.  and Herrig, T.,
Dynamical degrees of automorphisms on abelian varieties,
{\it Adv. Math.} {\bf 395} (2022), Paper No. 108082, 43 pp.



\bibitem{DTD12}
De Thélin, H. and 
 Dinh, T.-C.,
Dynamics of automorphisms on compact K{\"a}hler manifolds, 
{\it Adv. Math.} {\bf  229} (2012), no. 5, 2640-2655. 


\bibitem{Dinh05JGA}
Dinh, T.-C.,
Suites d’applications m{\'e}romorphes multivalu{\'e}es et courants laminaires,
{\it J. Geom. Anal.}
{\bf 15} (2005), no. 2, 207-227.

\bibitem{Dinh05}
Dinh, T.-C.,
Decay of correlations for Hénon maps,
{\it  Acta Math.} {\bf 195}
(2005), no. 2, 253-264.

\bibitem{DN06}
Dinh T.-C. and Nguyên V.-A.,
The mixed Hodge-Riemann bilinear relations for compact K\"ahler
manifolds, 
{\it Geom. Funct. Anal.} {\bf 16} (2006), 838-849.


\bibitem{DNV18}
Dinh, T.-C,
Nguyên, V.-A, and
 Vu, D.-V.,
  Super-potentials, densities of currents and number of periodic points for holomorphic maps,
  {\it Adv. Math.} {\bf 331} (2018), 874–907.

\bibitem{DS04}
Dinh, T.-C. and  Sibony, N.,
Regularization of currents and entropy,
{\it Ann. Sci. École Norm. Sup. (4)}
{\bf 37} (2004), no. 6, 959-971. 


\bibitem{DS05}
Dinh, T.-C. and Sibony, N., 
Green currents for holomorphic automorphisms of compact K\"ahler
manifolds,
{\it J. Amer. Math. Soc.} {\bf 18}
 (2005), 291-312.




\bibitem{DS10JAG}
Dinh, T.-C. and Sibony, N., 
Super-potentials for currents on compact K\"ahler manifolds and dynamics of automorphisms,
{\it J. Alg. Geom.}
{\bf 19}
(2010), 473-529.


\bibitem{DS10}
Dinh, T.-C. and Sibony, N.,
Exponential mixing for automorphisms on compact K\"ahler manifolds,
{\it Contemp. Math.} {\bf 532} (2010), 107-114.

\bibitem{DFL}
Dolgopyat, D., Fayad, B., and
 Liu, S.,
  Multiple Borel Cantelli Lemma in dynamics and MultiLog law for recurrence,
   {\it  preprint }  (2021),
   {\tt arXiv:2103.08382}.


\bibitem{DKRH}
Dolgopyat, D., Kanigowski, A., and Rodriguez-Hertz, F.,
Exponential mixing implies Bernoulli, 
{\it preprint} (2021), 
{\tt arXiv:2106.03147}.


\bibitem{Federer}
Federer, H.,
Geometric measure theory,
Classics in Mathematics,
Springer, 2014.

\bibitem{FT21}
Filip, S. and Tosatti, V.,
Kummer rigidity for $K3$ surface automorphisms via Ricci-flat metrics,
{\it Amer. J. Math.} {\bf 143} (2021), no. 5, 1431–1462. 

\bibitem{FT23}
Filip, S. and
 Tosatti, V.,
 Gaps in the support of canonical currents on projective $K3$ surfaces,
{\it preprint} (2023) {\tt arXiv:2302.08633}.





\bibitem{Gromov90}
Gromov, M.,
Convex sets and K\"ahler manifolds, Advances in Differential Geometry and Topology
World Sci. Publishing, Teaneck, NJ (1990), 1-38.

\bibitem{Gromov03}
Gromov M.,
On the entropy of holomorphic maps,
{\it Enseign. Math.(2)} {\bf 49} (2003), no. 3-4, 217-235.


 \bibitem{Guedj10}
Guedj V.,
Propri{\'e}t{\'e}s ergodiques des applications rationnelles, in 
{\it Quelques aspects des syst{\`e}mes dynamiques polynomiaux}, 
Panoramas et Synth. {\bf 30} (2010).


\bibitem{Khovanskii79}
Khovanskii, A.G.,
The geometry of convex polyhedra and algebraic geometry,
{\it Uspehi Mat.
Nauk.}
{\bf 34} (1979), no. 4, 160-161.




\bibitem{McM06}
McMullen, C. T.,
Dynamics on $K3$ surfaces: Salem numbers and Siegel disks,
{\it J. Reine Angew. Math.} {\bf 545} (2002), 201–233. 

\bibitem{McM07}
McMullen C. T.,
Dynamics on blowups of the projective plane,
{\it Publ. Math. Inst. Hautes {\'E}tudes Sci.}
{\bf 105} (2007), 49-89.

\bibitem{Oguiso09}
Oguiso K.,
A remark on Dynamical degrees of automorphisms of compact Hyper{k\"a}hler manifolds,
{\it Manuscripta Math.} {\bf 130} (2009), no. 1, 101-111.


\bibitem{Oguiso14}
Oguiso, K. 
Some aspects of explicit birational geometry inspired by complex dynamics,
{\it Proceedings of the International Congress of Mathematicians - Seoul 2014}. Vol. II, 695–721, Kyung Moon Sa, Seoul, 2014.



\bibitem{OT15}
Oguiso, K.
and Truong, T. T.,
Explicit examples of rational and Calabi-Yau threefolds with primitive automorphisms of positive entropy,
{\it  J. Math. Sci. Univ. Tokyo} {\bf 22} (2015), no. 1, 361–385.



\bibitem{Taflin11}
Taflin, J.,
Dynamique des endomorphismes holomorphes de l'espace projectif,
PhD thesis, Paris 6
(2011).


\bibitem{Teissier79}
Teissier, B.,
Du théor{\`e}me de l'index de Hodge aux in{\'e}galit{\'e}s isop{\'e}rim{\'e}triques,
{\it C. R. Acad.
Sci. Paris S{\'e}r. A-B} {\bf 288} (1979), no. 4, 287-289.



\bibitem{Triebel95}
Triebel, H.,
Interpolation theory, Function Spaces, Differential Operators, 2nd edition, Johann
Ambrosius Barth, Heidelberg, 1995.


\bibitem{Vigny15}
Vigny, G.,
Exponential decay of correlations for generic birational maps of $\mathbb P^k$,
{\it  Mathematische Annalen}  {\bf 362} 
(2015), 1033–1054.



\bibitem{Voisin02}
Voisin, C.,
Th{\'e}orie de Hodge et g{\'e}om{\'e}rie 
alg{\'e}brique complexe, Cours Sp{\'e}cialis{\'e}s, {\bf 10}, Soci{\'e}t{\'e}
Math{\'e}matique de France, Paris, 2002.

\bibitem{Wu22}
Wu, H.,
Exponential mixing property for H{\'e}non–Sibony maps of $\mathbb C^k$, 
{\it Ergodic Theory and Dynamical Systems} {\bf 42} (2022) no. 12, 3818-3830.



\bibitem{Yomdin87}
Yomdin, Y.,
Volume growth and entropy,
{\it  Israel J. Math.}
{\bf 57} (1987), no. 3, 285-300.


\end{thebibliography}
\end{document}